\documentclass[11pt,reqno]{amsart}
\usepackage{amsmath}
\usepackage{amsfonts}
\usepackage{amssymb,latexsym}
\usepackage{cite}
\usepackage[height=190mm,width=130mm]{geometry}

\theoremstyle{plain}
\newtheorem{theorem}{Theorem}
\newtheorem{lemma}{Lemma}
\newtheorem{corollary}{Corollary}

\theoremstyle{definition}
\newtheorem{definition}{Definition}
\theoremstyle{remark}
\newtheorem{remark}{Remark}

\numberwithin{equation}{section} 
\input{tcilatex}

\begin{document}
\title[On the r-circulant matrices with the generalized bi-periodic
Fibonacci numbers]{On the r-circulant matrices with the generalized
bi-periodic Fibonacci numbers}
\author{Mehmet DA\u{G}LI}
\address{Department of Mathematics, Amasya University, Amasya, 05000, Turkey}
\email{mehmet.dagli@amasya.edu.tr}
\author{Elif TAN}
\address{Department of Mathematics, Ankara University, Science Faculty,
06100 Tandogan Ankara, Turkey.}
\email{etan@ankara.edu.tr}
\author{Oktay OLMEZ}
\address{Department of Mathematics, Ankara University, Science Faculty,
06100 Tandogan Ankara, Turkey}
\email{oolmez@ankara.edu.tr}
\subjclass[2000]{ 15A60, 15B05, 11B39}
\keywords{r-circulant matrix, spectral norm, Fibonacci numbers, bi-periodic
Fibonacci numbers.}

\begin{abstract}
In this paper, we give upper and lower bounds for the spectral norms of
r-circulant matrices with the generalized bi-periodic Fibonacci numbers.
Moreover, we investigate the eigenvalues and determinants of these matrices.
\end{abstract}

\maketitle

\section{Introduction}

The \emph{generalized bi-periodic Fibonacci sequence }$\{w_{n}\}:=%
\{w_{n}(w_{0},w_{1};a,b)\}$ with arbitrary initial conditions $w_{0}$ and $%
w_{1}$, is defined by the recurrence relation 
\begin{equation}
w_{n}=a^{\zeta (n+1)}b^{\zeta (n)}w_{n-1}+w_{n-2},\quad n\geq 2  \label{1}
\end{equation}%
where $\zeta (n)=(1-(-1)^{n})/2$, and $a$ and $b$ be nonzero real numbers 
\cite{edson}. Note that $\zeta (n)$ returns to $0$ when $n$ is even, and to $%
1$ when $n$ is odd. Several well-known integer sequences are special cases
of this sequence. For example, this sequence is reduced to the bi-periodic
Fibonacci sequence $\{q_{n}\}$ for $w_{0}=0,w_{1}=1$, and to the bi-periodic
Lucas sequence $\{p_{n}\}$ for $w_{0}=2,w_{1}=b$. We refer to \cite{edson,
yayenie, panario, bilgici, tan1, tan, tan-leung} for basic properties of
these sequences and their generalizations.

For $n>0$, the Binet formula of the sequence $\{w_{n}\}$ can be written as 
\begin{equation}
w_{n}=\frac{a^{\zeta \left( n+1\right) }}{\left( ab\right) ^{\left\lfloor 
\frac{n}{2}\right\rfloor }}\left( \mathbf{A}\alpha ^{n}-\mathbf{B}\beta
^{n}\right) ,  \label{2}
\end{equation}%
where $\mathbf{A}=\frac{w_{1}-({\beta }/{a})w_{0}}{\alpha -\beta }$ and $%
\mathbf{B}=\frac{w_{1}-({\alpha }/{a})w_{0}}{\alpha -\beta }$ \cite%
{tan-leung, tan}. The numbers $\alpha =\frac{ab+\sqrt{a^{2}b^{2}+4ab}}{2}$
and $\beta =\frac{ab-\sqrt{a^{2}b^{2}+4ab}}{2}$ are the roots of the
polynomial $x^{2}-abx-ab$, and they satisfy the following equalities: 
\begin{equation*}
\alpha +\beta =ab,\quad \alpha -\beta =\sqrt{a^{2}b^{2}+4ab},\quad \alpha
\beta =-ab.
\end{equation*}%
Thus the Binet formulas for the sequences $\{q_{n}\}$ and $\{p_{n}\}$ are
given by 
\begin{equation*}
q_{n}=\frac{a^{\xi \left( n+1\right) }}{\left( ab\right) ^{\left\lfloor 
\frac{n}{2}\right\rfloor }}\left( \frac{\alpha ^{n}-\beta ^{n}}{\alpha
-\beta }\right) \quad \text{and}\quad p_{n}=\frac{a^{-\xi \left( n\right) }}{%
\left( ab\right) ^{\left\lfloor \frac{n}{2}\right\rfloor }}\left( \alpha
^{n}+\beta ^{n}\right) .
\end{equation*}%
The generating function of $\{w_{n}\}$ is given in \cite[Theorem 8]{edson}
as 
\begin{equation*}
G\left( x\right) =\frac{w_{0}+w_{1}x+\left( aw_{1}-\left( ab+1\right)
w_{0}\right) x^{2}+\left( bw_{0}-w_{1}\right) x^{3}}{1-\left( ab+2\right)
x^{2}+x^{4}}.
\end{equation*}

Our goal is to calculate the Frobenius norm, and find upper and lower bounds
on the spectral norm of $r-$circulant matrices whose entries are generalized
bi-periodic Fibonacci numbers. To this purpose, we review the background
material concerning the basic definitions and facts of $r-$circulant
matrices and matrix norms in the rest of this section.

Let $r\in \mathbb{C-}\left\{ 0\right\} $. An $n\times n$ matrix $C_{r}=%
\begin{bmatrix}
c_{ij}%
\end{bmatrix}%
$ with entries 
\begin{equation*}
c_{ij}=\left\{ 
\begin{array}{cc}
c_{j-i}, & j\geq i \\ 
rc_{n+j-i}, & j<i%
\end{array}%
\right.
\end{equation*}%
is called an $r-$\emph{circulant} matrix. In other words, $C_{r}$ has the
following form: 
\begin{eqnarray*}
&& \\
C_{r} &=&\left[ 
\begin{array}{cccccc}
c_{0} & c_{1} & c_{2} & \dots & c_{n-2} & c_{n-1} \\ 
rc_{n-1} & c_{0} & c_{1} & \dots & c_{n-3} & c_{n-2} \\ 
rc_{n-2} & rc_{n-1} & c_{0} & \dots & c_{n-4} & c_{n-3} \\ 
\vdots & \vdots & \vdots & \ddots & \vdots & \vdots \\ 
rc_{2} & rc_{3} & rc_{4} & \dots & c_{0} & c_{1} \\ 
rc_{1} & rc_{2} & rc_{3} & \dots & rc_{n-1} & c_{0}%
\end{array}%
\right] . \\
&&
\end{eqnarray*}%
For the simplicity, we denote $C_{r}$ by $\func{circ}_{r,n}\left[
c_{0},c_{1},\dots ,c_{n-1}\right] $. Note that for $r=1,$ it reduces to
circulant matrix $C.$ The eigenvalues of $C_{r}$ are given as%
\begin{equation}
\lambda _{j}\left( C_{r}\right) =\sum_{k=0}^{n-1}c_{k}\left( \rho \omega
^{-j}\right) ^{k},  \label{eigenvalue}
\end{equation}%
where $\rho $ is any $n$th root of $r$, $\omega $ is any $n$th root of
unity, and $j=0,1,\ldots ,n-1.$ For details, we refer to \cite[Lemma 4]%
{cline}.

Let $A=%
\begin{bmatrix}
a_{ij}%
\end{bmatrix}%
$ be an $m\times n$ matrix. The Frobenius norm (also known as
Hilbert-Schmidt norm or Schur norm) $\left\Vert A\right\Vert _{F}$ of $A$ is
the square root of the sum of the squares of the absolute values of all
entries of $A$. That is, 
\begin{equation*}
\left\Vert A\right\Vert _{F}:=\sqrt{\sum_{i=1}^{m}\sum_{j=1}^{n}\left\vert
a_{ij}\right\vert ^{2}}.
\end{equation*}%
Note that the Frobenius norm measures the size of a matrix. Another norm
used in matrix applications is the spectral norm defined as 
\begin{equation*}
\left\Vert A\right\Vert _{2}:=\sqrt{\lambda _{max}\left( A^{\ast }A\right) },
\end{equation*}%
where $\lambda _{max}\left( A^{\ast }A\right) $ denotes the largest
eigenvalue of $A^{\ast }A$. Here, $A^{\ast }$ is the conjugate transpose of $%
A$. The following inequality by Zielke in \cite{zielke} provides a
relationship between Frobenius and spectral norms:%
\begin{equation}
\frac{1}{\sqrt{n}}\left\Vert A\right\Vert _{F}\leq \left\Vert A\right\Vert
_{2}\leq \left\Vert A\right\Vert _{F}\qquad \text{and}\qquad \left\Vert
A\right\Vert _{2}\leq \left\Vert A\right\Vert _{F}\leq \sqrt{n}\left\Vert
A\right\Vert _{2}.  \label{*}
\end{equation}%
Note that the Frobenius norm is an upper bound on the spectral norm. An
important tool for finding bounds on the matrix norms is the Hadamard
product. For $m\times n$ matrices $A=%
\begin{bmatrix}
a_{ij}%
\end{bmatrix}%
$ and $B=%
\begin{bmatrix}
b_{ij}%
\end{bmatrix}%
$, the Hadamard product of $A$ and $B$ is defined as $A\circ B:=%
\begin{bmatrix}
a_{ij}b_{ij}%
\end{bmatrix}%
$. It is simply the entry wise multiplication of $A$ and $B$. In \cite%
{mathias}, Mathias proved that 
\begin{equation}
\left\Vert A\circ B\right\Vert _{2}\leq \left\Vert A\right\Vert
_{2}\left\Vert B\right\Vert _{2}\qquad \text{and}\qquad \left\Vert A\circ
B\right\Vert _{2}\leq r_{1}\left( A\right) c_{1}\left( B\right) ,  \label{**}
\end{equation}%
where 
\begin{equation*}
r_{1}\left( A\right) =\max_{1\leq i\leq m}\sqrt{\sum_{j=1}^{n}\left\vert
a_{ij}\right\vert ^{2}}\qquad \text{and}\qquad c_{1}\left( B\right)
=\max_{1\leq j\leq n}\sqrt{\sum_{i=1}^{m}\left\vert b_{ij}\right\vert ^{2}}.
\end{equation*}%
Note that $r_{1}(A)$ is the maximum row length norm of $A$, and $c_{1}(B)$
is the maximum column length norm of $B$. See also \cite{horn}.

There has been many works related to the properties of circulant matrices, $%
r-$circulant matrices and special type of circulant matrices with
Fibonacci-like numbers. In particular, Solak \cite{solak, solak1} obtained
some bounds for the spectral norms of circulant matrices whose entries are
Fibonacci and Lucas numbers. For the spectral norms of $r-$circulant
matrices with Fibonacci and Lucas numbers, Shen and Cen\cite{shen} gave some
bounds which generalized the results in \cite{solak}. Nalli and Sen \cite%
{nalli} investigate the norms of circulant matrices with generalized
Fibonacci numbers. Alptekin et al. \cite{alptekin} obtained the spectral
norms and eigenvalues of circulant matrices whose entries are Horadam
numbers. In \cite{yazlik}, Yazlik and Taskara found upper and lower bounds
on the norms of an $r-$circulant matrix with the generalized $k-$Horadam
numbers. Also they established the determinant and eigenvalues of this
matrix. For related works, we refer to \cite{lind, nalli, amara, bahsi,
pentti}. Recently, Kome and Yazlik \cite{kome} have obtained some upper and
lower bounds for the spectral norms of the $r-$circulant matrices whose
entries are bi-periodic Fibonacci and Lucas numbers. Here, we examine upper
and lower bounds for the spectral norms of $r-$circulant matrices with the
generalized bi-periodic Fibonacci numbers. Also, we obtain the eigenvalues
and determinants of these matrices.

\section{Main results}

Throughout this section we let $a,b$, and $w_{1}$ to be positive integers
and let $w_{0}$ to be a nonnegative integer unless otherwise is stated.

\begin{definition}
The $r-$circulant matrix $W_{r}$ with the generalized bi-periodic Fibonacci
numbers is defined as%
\begin{eqnarray*}
&& \\
W_{r} &:&=\!{\footnotesize \left[ \!%
\begin{array}{ccccc}
\left( \frac{a}{b}\right) ^{\frac{\zeta \left( 0\right) }{2}}w_{0} & \left( 
\frac{a}{b}\right) ^{\frac{\zeta \left( 1\right) }{2}}w_{1} & \left( \frac{a%
}{b}\right) ^{\frac{\zeta \left( 2\right) }{2}}w_{2} & \cdots & \left( \frac{%
a}{b}\right) ^{\frac{\zeta \left( n-1\right) }{2}}w_{n-1} \\[6pt] 
r\left( \frac{a}{b}\right) ^{\frac{\zeta \left( n-1\right) }{2}}w_{n-1} & 
\left( \frac{a}{b}\right) ^{\frac{\zeta \left( 0\right) }{2}}w_{0} & \left( 
\frac{a}{b}\right) ^{\frac{\zeta \left( 1\right) }{2}}w_{1} & \cdots & 
\left( \frac{a}{b}\right) ^{\frac{\zeta \left( n-2\right) }{2}}w_{n-2} \\%
[6pt] 
r\left( \frac{a}{b}\right) ^{\frac{\zeta \left( n-2\right) }{2}}w_{n-2} & 
r\left( \frac{a}{b}\right) ^{\frac{\zeta \left( n-1\right) }{2}}w_{n-1} & 
\left( \frac{a}{b}\right) ^{\frac{\zeta \left( 0\right) }{2}}w_{0} & \cdots
& \left( \frac{a}{b}\right) ^{\frac{\zeta \left( n-3\right) }{2}}w_{n-3} \\%
[6pt] 
\vdots & \vdots & \vdots & \ddots & \vdots \\[6pt] 
r\left( \frac{a}{b}\right) ^{\frac{\zeta \left( 1\right) }{2}}w_{1} & 
r\left( \frac{a}{b}\right) ^{\frac{\zeta \left( 2\right) }{2}}w_{2} & 
r\left( \frac{a}{b}\right) ^{\frac{\zeta \left( 3\right) }{2}}w_{3} & \cdots
& \left( \frac{a}{b}\right) ^{\frac{\zeta \left( 0\right) }{2}}w_{0}%
\end{array}%
\!\right] .} \\
&&
\end{eqnarray*}%
In accord with the notational convention introduced in the previous section,
we can write $W_{r}$ as follows:%
\begin{equation*}
W_{r}=\func{circ}_{r,n}\left[ \left( \frac{a}{b}\right) ^{\frac{\zeta \left(
0\right) }{2}}w_{0},\,\left( \frac{a}{b}\right) ^{\frac{\zeta \left(
1\right) }{2}}w_{1}\,,\ldots ,\,\left( \frac{a}{b}\right) ^{\frac{\zeta
\left( n-1\right) }{2}}w_{n-1}\right] .
\end{equation*}
\end{definition}

\begin{lemma}
\label{Sums-of-squares1}For $n>1$, we have 
\begin{eqnarray*}
\sum_{k=1}^{n}\left( \frac{a}{b}\right) ^{\zeta \left( k\right) }w_{k}^{2}
&=&\frac{1}{b}\left( w_{n}w_{n+1}-w_{0}w_{1}\right) . \\
&&
\end{eqnarray*}
\end{lemma}

\begin{proof}
Recall that the Binet formula of the sequence $\{w_{k}\}$ is given by 
\begin{equation*}
w_{k}=\frac{a^{\zeta \left( k+1\right) }}{\left( ab\right) ^{\left\lfloor 
\frac{k}{2}\right\rfloor }}\left( \mathbf{A}\alpha ^{k}-\mathbf{B}\beta
^{k}\right) .
\end{equation*}%
Since $\zeta (n)+\zeta (n+1)=1$ and $\lfloor n/2\rfloor +\lfloor
(n+1)/2\rfloor =n$, we have 
\begin{eqnarray}
w_{n}w_{n+1} &=&\frac{a^{\zeta \left( n+1\right) +\zeta \left( n\right) }}{%
\left( ab\right) ^{\left\lfloor \frac{n}{2}\right\rfloor +\left\lfloor \frac{%
n+1}{2}\right\rfloor }}\left( \mathbf{A}\alpha ^{n}-\mathbf{B}\beta
^{n}\right) \left( \mathbf{A}\alpha ^{n+1}-\mathbf{B}\beta ^{n+1}\right) 
\notag \\
&=&\frac{a}{\left( ab\right) ^{n}}\left[ \mathbf{A}^{2}\alpha ^{2n+1}-%
\mathbf{AB}\left( \alpha \beta \right) ^{n}\left( \alpha +\beta \right) +%
\mathbf{B}^{2}\beta ^{2n+1}\right]  \notag \\
&=&\frac{a}{\left( ab\right) ^{n}}\left[ \mathbf{A}^{2}\alpha ^{2n+1}+%
\mathbf{B}^{2}\beta ^{2n+1}-\mathbf{AB}\left( -ab\right) ^{n}\left(
ab\right) \right] .  \label{Eqn1}
\end{eqnarray}%
On the other hand, we have 
\begin{equation*}
\begin{split}
w_{k}^{2}& =\frac{a^{2\zeta \left( k+1\right) }}{\left( ab\right)
^{2\left\lfloor \frac{k}{2}\right\rfloor }}\left[ \mathbf{A}\alpha ^{k}-%
\mathbf{B}\beta ^{k}\right] ^{2} \\
& =\frac{a^{2\zeta \left( k+1\right) }}{\left( ab\right) ^{2\left\lfloor 
\frac{k}{2}\right\rfloor }}\left[ \mathbf{A}^{2}\alpha ^{2k}+\mathbf{B}%
^{2}\beta ^{2k}-2\mathbf{AB}\left( \alpha \beta \right) ^{k}\right] .
\end{split}%
\end{equation*}%
Now, if $k$ is even, we get 
\begin{equation*}
w_{k}^{2}=\frac{a^{2}}{\left( ab\right) ^{k}}\left[ \mathbf{A}^{2}\alpha
^{2k}+\mathbf{B}^{2}\beta ^{2k}-2\mathbf{AB}\left( \alpha \beta \right) ^{k}%
\right] ,
\end{equation*}%
and if $k$ is odd, we get 
\begin{equation*}
w_{k}^{2}=\frac{ab}{\left( ab\right) ^{k}}\left[ \mathbf{A}^{2}\alpha ^{2k}+%
\mathbf{B}^{2}\beta ^{2k}-2\mathbf{AB}\left( \alpha \beta \right) ^{k}\right]
.
\end{equation*}%
Since $\alpha \beta =-ab$, and 
\begin{equation*}
a^{2}\left( \frac{b}{a}\right) ^{\zeta (k)}=\left\{ 
\begin{array}{cl}
a^{2}, & \text{if $k$ is even} \\ 
ab, & \text{if $k$ is odd}%
\end{array}%
\right.
\end{equation*}%
we can write 
\begin{equation*}
w_{k}^{2}=a^{2}\left( \frac{b}{a}\right) ^{\zeta \left( k\right) }\left[ 
\mathbf{A}^{2}\left( \frac{\alpha ^{2}}{ab}\right) ^{k}+\mathbf{B}^{2}\left( 
\frac{\beta ^{2}}{ab}\right) ^{k}-2\mathbf{AB}\left( -1\right) ^{k}\right] ,
\end{equation*}%
or equivalently, 
\begin{equation}
a^{-2}\left( \frac{a}{b}\right) ^{\zeta \left( k\right) }w_{k}^{2}=\mathbf{A}%
^{2}\left( \frac{\alpha ^{2}}{ab}\right) ^{k}+\mathbf{B}^{2}\left( \frac{%
\beta ^{2}}{ab}\right) ^{k}-2\mathbf{AB}\left( -1\right) ^{k}.  \label{Eqn2}
\end{equation}%
By using the geometric sum formula, it can be seen that 
\begin{equation*}
\sum_{k=1}^{n}\left( \frac{\alpha ^{2}}{ab}\right) ^{k}=\frac{\alpha ^{2n+1}%
}{(ab)^{n+1}}-\frac{\alpha }{ab}\quad \text{and}\quad \sum_{k=1}^{n}\left( 
\frac{\beta ^{2}}{ab}\right) ^{k}=\frac{\beta ^{2n+1}}{(ab)^{n+1}}-\frac{%
\beta }{ab}.
\end{equation*}%
Now we take the summation of both sides of Equation \ref{Eqn2} from $1$ to $%
k $:

\begin{equation*}
\begin{split}
& \sum_{k=1}^{n}a^{-2}\left( \frac{a}{b}\right) ^{\zeta \left( k\right)
}w_{k}^{2} \\
& =\mathbf{A}^{2}\sum_{k=1}^{n}\left( \frac{\alpha ^{2}}{ab}\right) ^{k}+%
\mathbf{B}^{2}\sum_{k=1}^{n}\left( \frac{\beta ^{2}}{ab}\right) ^{k}-\mathbf{%
AB}\sum_{k=1}^{n}2\left( -1\right) ^{k} \\
& =\mathbf{A}^{2}\left[ \frac{\alpha ^{2n+1}}{(ab)^{n+1}}-\frac{\alpha }{ab}%
\right] +\mathbf{B}^{2}\left[ \frac{\beta ^{2n+1}}{(ab)^{n+1}}-\frac{\beta }{%
ab}\right] -\mathbf{AB}\left[ \left( -1\right) ^{n}-1\right] \\
& =\frac{1}{ab}\left[ \mathbf{A}^{2}\alpha ^{2n+1}\left( \frac{1}{ab}\right)
^{n}+\mathbf{B}^{2}\beta ^{2n+1}\left( \frac{1}{ab}\right) ^{n}-\mathbf{A}%
^{2}\alpha -\mathbf{B}^{2}\beta -\mathbf{AB}ab\left[ \left( -1\right) ^{n}-1%
\right] \right] .
\end{split}%
\end{equation*}

\noindent By taking Equation \ref{Eqn1} into account in the last line of the
equation above, we get the desired result: 
\begin{eqnarray*}
\sum_{k=1}^{n}a^{-2}\left( \frac{a}{b}\right) ^{\zeta \left( k\right)
}w_{k}^{2} &=& \frac{1}{ab}\frac{1}{a}\left( w_{n}w_{n+1}\right) + \mathbf{AB%
}-\frac{\mathbf{A}^{2}\alpha }{ab}-\frac{\mathbf{B}^{2}\beta }{ab} \\
&=&\frac{1}{ab}\left[ \frac{1}{a}\left( w_{n}w_{n+1}\right) +\mathbf{AB}
\left( ab\right) -\left( \mathbf{A}^{2}\alpha +\mathbf{B}^{2}\beta \right) %
\right] \\
&=&\frac{1}{ab}\left[ \frac{1}{a}\left( w_{n}w_{n+1}\right) -\frac{1}{a}
w_{0}w_{1}\right] \\
&=&\frac{1}{a^{2}b}\left[ w_{n}w_{n+1} -w_{0}w_{1}\right] .
\end{eqnarray*}
\end{proof}

An immediate result of Lemma \ref{Sums-of-squares1} is the following.

\begin{corollary}
\label{Sums-of-squares2} For $n>0$, the following equality holds: 
\begin{equation*}
\sum_{k=0}^{n-1}\left( \frac{a}{b}\right) ^{\zeta \left( k\right) }w_{k}^{2}=%
\frac{1}{b}\left( w_{n}w_{n-1}-w_{0}w_{1}+bw_{0}^{2}\right) .
\end{equation*}
\end{corollary}

\begin{remark}
If we take the initial conditions $w_{0}=0$ and $w_{1}=1$, we get the
identity%
\begin{equation*}
\sum_{k=1}^{n}\left( \frac{a}{b}\right) ^{\zeta \left( k\right) }q_{k}^{2}=%
\frac{1}{b}q_{n}q_{n+1},
\end{equation*}%
given in \cite{yayenie}. Similary, with the initial conditions $w_{0}=2$ and 
$w_{1}=b$, we get the identity%
\begin{equation*}
\sum_{k=1}^{n}\left( \frac{a}{b}\right) ^{\zeta \left( k\right) }p_{k}^{2}=%
\frac{1}{b}p_{n}p_{n+1}-2,
\end{equation*}%
given in \cite{tan1}.
\end{remark}

Now we are ready to provide lower and upper bounds on the spectral norm of
an $r-$circulant matrix with generalized bi-periodic Fibonacci numbers%
\begin{equation*}
W_{r}=\func{circ}_{r,n}\left[ \left( \frac{a}{b}\right) ^{\frac{\zeta \left(
0\right) }{2}}w_{0},\left( \frac{a}{b}\right) ^{\frac{\zeta \left( 1\right) 
}{2}}w_{1},\ldots ,\left( \frac{a}{b}\right) ^{\frac{\zeta \left( n-1\right) 
}{2}}w_{n-1}\right] .
\end{equation*}%
But first, we calculate the Frobenius norm.

\begin{lemma}
\label{Frobenus-norm} The Frobenius norm of $W_r$ is given by 
\begin{eqnarray*}
\left\Vert W_{r} \right\Vert _{F} &=&\sqrt{\sum_{k=0}^{n-1}\left( n+k\left(
\left\vert r\right\vert ^{2}-1\right) \right) \left( \frac{a}{b}\right)
^{\zeta \left( k\right) }w_{k}^{2}}.
\end{eqnarray*}
\end{lemma}

\begin{proof}
By using Lemma \ref{Sums-of-squares1} and Corollary \ref{Sums-of-squares2},
it is clear that 
\begin{eqnarray*}
\left\Vert W_{r}\right\Vert _{F} &=&\sqrt{\sum_{k=0}^{n-1}\left( n-k\right)
\left( \frac{a}{b}\right) ^{\zeta \left( k\right)
}w_{k}^{2}+\sum_{k=1}^{n-1}k\left\vert r\right\vert ^{2}\left( \frac{a}{b}%
\right) ^{\zeta \left( k\right) }w_{k}^{2}} \\
&=&\sqrt{\sum_{k=0}^{n-1}\left( n+k\left( \left\vert r\right\vert
^{2}-1\right) \right) \left( \frac{a}{b}\right) ^{\zeta \left( k\right)
}w_{k}^{2}}.
\end{eqnarray*}
\end{proof}

\begin{theorem}
\label{t1}Let $\Delta :=w_{n-1}w_{n}-w_{0}w_{1}+bw_{0}^{2}$. Then the
following inequalities hold for the $r$-circulant matrix $W_{r}$:

\noindent $(i)$ If $|r|\geq 1$, then%
\begin{equation*}
\sqrt{\frac{\Delta }{b}}\leq \left\Vert W_{r}\right\Vert _{2}\leq \sqrt{%
\left( \left( n-1\right) \left\vert r\right\vert ^{2}+1\right) \frac{\Delta 
}{b}}.
\end{equation*}%
$\noindent (ii)$ If $|r|<1$, then 
\begin{equation*}
\left\vert r\right\vert \sqrt{\frac{\Delta }{b}}\leq \left\Vert
W_{r}\right\Vert _{2}\leq \sqrt{n\frac{\Delta }{b}}.
\end{equation*}
\end{theorem}

\begin{proof}
\noindent $(i)$ Let $\left\vert r\right\vert \geq 1.$ From Corollary \ref%
{Sums-of-squares2} and Lemma \ref{Frobenus-norm}, we have 
\begin{equation*}
\left\Vert W_{r}\right\Vert _{F}\geq \sqrt{\sum_{k=0}^{n-1}n\left( \frac{a}{b%
}\right) ^{\zeta \left( k\right) }w_{k}^{2}}=\sqrt{n\frac{\Delta }{b}}.
\end{equation*}%
Therefore, we can write 
\begin{equation*}
\frac{1}{\sqrt{n}}\left\Vert W_{r}\right\Vert _{F}\geq \sqrt{\frac{\Delta }{b%
}}.
\end{equation*}

From the inequality (\ref{*}), we obtain 
\begin{equation*}
\sqrt{\frac{\Delta }{b}}\leq \left\Vert W_{r}\right\Vert _{2}.
\end{equation*}

In order to provide an upper bound for the spectral norm of $W_{r}$, we
consider the following matrices: 
\begin{eqnarray*}
&& \\
U &:&=\left[ 
\begin{array}{ccccc}
1 & 1 & \cdots & 1 & 1 \\ 
r & 1 & \cdots & 1 & 1 \\ 
\vdots & \vdots & \ddots & \vdots & \vdots \\ 
r & r & \cdots & 1 & 1 \\ 
r & r & \cdots & r & 1%
\end{array}%
\right] , \\
&&
\end{eqnarray*}%
\begin{eqnarray*}
&& \\
W &=&{\small \left[ 
\begin{array}{ccccc}
\left( \frac{a}{b}\right) ^{\frac{\zeta \left( 0\right) }{2}}w_{0} & \left( 
\frac{a}{b}\right) ^{\frac{\zeta \left( 1\right) }{2}}w_{1} & \left( \frac{a%
}{b}\right) ^{\frac{\zeta \left( 2\right) }{2}}w_{2} & \cdots & \left( \frac{%
a}{b}\right) ^{\frac{\zeta \left( n-1\right) }{2}}w_{n-1} \\ 
\left( \frac{a}{b}\right) ^{\frac{\zeta \left( n-1\right) }{2}}w_{n-1} & 
\left( \frac{a}{b}\right) ^{\frac{\zeta \left( 0\right) }{2}}w_{0} & \left( 
\frac{a}{b}\right) ^{\frac{\zeta \left( 1\right) }{2}}w_{1} & \cdots & 
\left( \frac{a}{b}\right) ^{\frac{\zeta \left( n-2\right) }{2}}w_{n-2} \\ 
\left( \frac{a}{b}\right) ^{\frac{\zeta \left( n-2\right) }{2}}w_{n-2} & 
\left( \frac{a}{b}\right) ^{\frac{\zeta \left( n-1\right) }{2}}w_{n-1} & 
\left( \frac{a}{b}\right) ^{\frac{\zeta \left( 0\right) }{2}}w_{0} & \cdots
& \left( \frac{a}{b}\right) ^{\frac{\zeta \left( n-3\right) }{2}}w_{n-3} \\ 
\vdots & \vdots & \vdots & \ddots & \vdots \\ 
\left( \frac{a}{b}\right) ^{\frac{\zeta \left( 1\right) }{2}}w_{1} & \left( 
\frac{a}{b}\right) ^{\frac{\zeta \left( 2\right) }{2}}w_{2} & \left( \frac{a%
}{b}\right) ^{\frac{\zeta \left( 3\right) }{2}}w_{3} & \cdots & \left( \frac{%
a}{b}\right) ^{\frac{\zeta \left( 0\right) }{2}}w_{0}%
\end{array}%
\right] }. \\
&&
\end{eqnarray*}

Note that $W=\func{circ}_{1,n}\left[ \left( \frac{a}{b}\right) ^{\frac{\zeta
\left( 0\right) }{2}}w_{0},\,\left( \frac{a}{b}\right) ^{\frac{\zeta \left(
1\right) }{2}}w_{1}\,,\ldots ,\,\left( \frac{a}{b}\right) ^{\frac{\zeta
\left( n-1\right) }{2}}w_{n-1}\right] .$ It is clear that $W_{r}=U\circ W$
where $\circ $ denotes the Hadamard product. Now we calculate the maximum
row length norm $r_{1}(U)$ of $U$, and maximum column length norm $c_{1}(W)$
of $W$. Since $|r|\geq 1$, we have

\begin{equation*}
r_{1}\left( U\right) =\max_{1\leq i\leq n}\sqrt{\sum_{j=1}^{n}\left\vert
u_{ij}\right\vert ^{2}}=\sqrt{\sum_{j=1}^{n}\left\vert u_{nj}\right\vert ^{2}%
}=\sqrt{\left( n-1\right) \left\vert r\right\vert ^{2}+1},
\end{equation*}%
and%
\begin{equation*}
c_{1}\left( W\right) =\max_{1\leq j\leq n}\sqrt{\sum_{i=1}^{n}\left\vert
w_{ij}\right\vert ^{2}}=\sqrt{\sum_{k=0}^{n-1}\left( \frac{a}{b}\right)
^{\zeta \left( k\right) }w_{k}^{2}}=\sqrt{\frac{\Delta }{b}}.\text{ \ \ \ \
\ \ \ }
\end{equation*}%
Using the above quantities, we obtain 
\begin{equation*}
\left\Vert W_{r}\right\Vert _{2}=\left\Vert U\circ W\right\Vert _{2} \\
\leq r_{1}\left( U\right) c_{1}\left( W\right) \\
=\sqrt{\left( (n-1)|r|^{2}+1\right) \frac{\Delta }{b}}.
\end{equation*}

\noindent $(ii)$ Let $|r|<1.$ Suppose $k$ is an integer with $0\leq k\leq
n-1 $. Since $|r|^{2}-1<0$, the minimum of $n+k(|r|^{2}-1)$ is achieved when 
$k=n-1$. So, for $k=n-1$ we have $n+k(|r|^{2}-1)=n|r|^{2}-|r|^{2}+1\geq
n|r|^{2}$. Then it follows that 
\begin{equation*}
n+k(|r|^{2}-1)\geq n|r|^{2}
\end{equation*}%
for each $k$ with $0\leq k\leq n-1$. Therefore, we can write 
\begin{eqnarray*}
\left\Vert W_{r}\right\Vert _{F} &=&\sqrt{\sum_{k=0}^{n-1}\left( n+k\left(
\left\vert r\right\vert ^{2}-1\right) \right) \left( \frac{a}{b}\right)
^{\zeta \left( k\right) }w_{k}^{2}} \\
&\geq &\sqrt{\sum_{k=0}^{n-1}n\left\vert r\right\vert ^{2}\left( \frac{a}{b}%
\right) ^{\zeta \left( k\right) }w_{k}^{2}}.
\end{eqnarray*}%
Then it follows that 
\begin{equation*}
\frac{1}{\sqrt{n}}\left\Vert W_{r}\right\Vert _{F}\geq \left\vert
r\right\vert \sqrt{\frac{\Delta }{b}}.
\end{equation*}%
Using the inequality (\ref{*}) again, we get 
\begin{equation*}
\left\Vert W_{r}\right\Vert _{2}\geq \left\vert r\right\vert \sqrt{\frac{%
\Delta }{b}}.
\end{equation*}

In order to provide an upper bound for the spectral norm of $W_{r}$, we
consider the matrices $U$ and $V$ defined above. Note that 
\begin{equation*}
r_{1}\left( U\right) =\max_{1\leq i\leq n}\sqrt{\sum_{j=1}^{n}\left\vert
u_{ij}\right\vert ^{2}}=\sqrt{\sum_{j=1}^{n}\left\vert u_{1j}\right\vert ^{2}%
}=\sqrt{n},\text{ \ \ \ \ \ \ \ \ \ }
\end{equation*}%
and 
\begin{equation*}
c_{1}\left( W\right) =\max_{1\leq j\leq n}\sqrt{\sum_{i=1}^{n}\left\vert
w_{ij}\right\vert ^{2}}=\sqrt{\sum_{k=0}^{n-1}\left( \frac{a}{b}\right)
^{\zeta \left( k\right) }w_{k}^{2}}=\sqrt{\frac{\Delta }{b}}.
\end{equation*}%
In conclusion, 
\begin{equation*}
\left\Vert W_{r}\right\Vert _{2}=\left\Vert U\circ W\right\Vert _{2}\leq
r_{1}\left( U\right) c_{1}\left( W\right) =\sqrt{n\frac{\Delta }{b}}.
\end{equation*}
\end{proof}

\begin{remark}
From Theorem \ref{t1}, we have the following results:

If we take $w_{0}=0,w_{1}=1,$ then we get\noindent\ the upper and lower
bounds for $r-$circulant matrix with bi-periodic Fibonacci numbers as:%
\begin{equation*}
\left. 
\begin{array}{c}
\sqrt{\frac{q_{n-1}q_{n}}{b}}\leq \left\Vert W_{r}\right\Vert _{2}\leq \sqrt{%
\left( \left( n-1\right) \left\vert r\right\vert ^{2}+1\right) \frac{%
q_{n-1}q_{n}}{b}},|r|\geq 1 \\ 
\left\vert r\right\vert \sqrt{\frac{q_{n-1}q_{n}}{b}}\leq \left\Vert
W_{r}\right\Vert _{2}\leq \sqrt{n\frac{q_{n-1}q_{n}}{b}}\text{ \ \ \ \ \ \ \
\ \ \ \ \ \ \ \ \ \ \ \ \ },|r|<1\text{\ }.\text{\ }%
\end{array}%
\right. \text{ \ \ \ \ \ \ \ \ \ \ \ \ \ \ \ \ \ \ \ \ \ \ \ \ \ \ \ \ \ \ \
\ \ \ \ \ \ \ \ \ \ \ \ \ \ \ \ \ \ \ \ \ \ \ \ \ \ \ \ \ \ \ \ \ \ \ \ \ \
\ \ \ \ \ \ \ \ \ \ \ \ \ \ \ \ \ \ \ \ \ \ \ \ \ \ \ \ \ \ \ \ \ \ \ \ \ \
\ \ \ \ \ \ }
\end{equation*}

If we take $w_{0}=2,w_{1}=b,$ then we get\noindent\ the upper and lower
bounds for $r-$circulant matrix with bi-periodic Lucas numbers as:%
\begin{equation*}
\left. 
\begin{array}{c}
\sqrt{\frac{p_{n-1}p_{n}}{b}+2}\leq \left\Vert W_{r}\right\Vert _{2}\leq 
\sqrt{\left( \left( n-1\right) \left\vert r\right\vert ^{2}+1\right) \left( 
\frac{p_{n-1}p_{n}}{b}+2\right) },|r|\geq 1 \\ 
\left\vert r\right\vert \sqrt{\frac{p_{n-1}p_{n}}{b}+2}\leq \left\Vert
W_{r}\right\Vert _{2}\leq \sqrt{n\left( \frac{p_{n-1}p_{n}}{b}+2\right) }%
\text{ \ \ \ \ \ \ \ \ \ \ \ \ \ \ \ \ \ \ \ \ },|r|<1.\text{\ \ \ }%
\end{array}%
\right. \text{ \ \ \ \ \ \ \ \ \ \ \ \ \ \ \ \ \ \ \ \ \ \ \ \ \ \ \ \ \ \ \
\ \ \ \ \ \ \ \ \ \ \ \ \ \ \ \ \ \ \ \ \ \ \ \ \ \ \ \ \ \ \ \ \ \ \ \ \ \
\ \ \ \ \ \ \ \ \ \ \ \ \ \ \ \ \ \ \ \ \ \ \ \ \ \ \ \ \ }
\end{equation*}

If we take $a=b=1,$ then we get\noindent\ the upper and lower bounds for $r-$%
circulant matrix with generalized Fibonacci numbers as:%
\begin{equation*}
\left. 
\begin{array}{c}
\sqrt{w_{n-1}w_{n}-w_{0}w_{1}+w_{0}^{2}}\leq \left\Vert W_{r}\right\Vert
_{2}\leq \sqrt{\left( \left( n-1\right) \left\vert r\right\vert
^{2}+1\right) \left( w_{n-1}w_{n}-w_{0}w_{1}+w\right) },|r|\geq 1 \\ 
\text{ \ }\left\vert r\right\vert \sqrt{w_{n-1}w_{n}-w_{0}w_{1}+w}\leq
\left\Vert W_{r}\right\Vert _{2}\leq \sqrt{n\left(
w_{n-1}w_{n}-w_{0}w_{1}+w\right) }\text{ \ \ \ \ \ \ \ \ \ \ \ \ \ \ \ \ \ \
\ },|r|<1.\text{\ \ \ }%
\end{array}%
\right. \text{ \ \ \ \ \ \ \ \ \ \ \ \ \ \ \ \ \ \ \ \ \ \ \ \ \ \ \ \ \ \ \
\ \ \ \ \ \ \ \ \ \ \ \ \ \ \ \ \ \ \ \ \ \ \ \ \ \ \ \ \ \ \ \ \ \ \ \ \ }
\end{equation*}
\end{remark}

\begin{theorem}
The eigenvalues of $W_{r}$ are 
\begin{equation*}
\lambda _{j}\left( W_{r}\right) =\frac{\left( \frac{a}{b}\right) ^{\frac{%
\zeta \left( n\right) }{2}}rw_{n}-w_{0}+\rho \omega ^{-j}\left( \left( \frac{%
a}{b}\right) ^{^{\frac{\zeta \left( n+1\right) }{2}}}rw_{n-1}+\left( \frac{a%
}{b}\right) ^{^{\frac{1}{2}}}\left( bw_{0}-w_{1}\right) \right) }{\rho
^{2}\omega ^{-2j}+\left( ab\right) ^{\frac{1}{2}}\rho \omega ^{-j}-1},
\end{equation*}%
where $\rho $ is any $n$th root of $r$, $\omega $ is any $n$th root of
unity, and $j=0,1,\ldots ,n-1.$
\end{theorem}

\begin{proof}
From (\ref{eigenvalue}), we have%
\begin{equation*}
\lambda _{j}\left( W_{r}\right) =\sum_{k=0}^{n-1}\left( \frac{a}{b}\right) ^{%
\frac{\zeta \left( k\right) }{2}}w_{k}\rho ^{k}\omega ^{-kj}\text{ \ \ \ \ \
\ \ \ \ \ \ \ \ \ \ \ \ \ \ \ \ \ \ \ \ \ \ \ \ \ \ \ \ \ \ \ \ \ \ \ \ \ \
\ \ \ \ \ \ \ \ \ \ \ \ \ \ \ \ \ \ \ \ \ \ \ \ \ \ \ \ \ \ \ \ \ \ \ \ \ \
\ \ \ \ \ \ \ \ \ \ \ \ \ \ \ \ \ \ \ \ \ \ \ \ \ \ \ \ \ \ \ \ \ \ \ \ \ \
\ \ \ \ \ \ \ \ \ \ \ \ \ \ \ \ \ \ \ \ \ \ \ \ \ \ \ \ \ \ \ \ \ \ \ \ \ \ }
\end{equation*}%
\begin{equation*}
=\mathbf{A}a\sum_{k=0}^{n-1}\left( \frac{\alpha \rho \omega ^{-j}}{\left(
ab\right) ^{\frac{1}{2}}}\right) ^{k}-\mathbf{B}a\sum_{k=0}^{n-1}\left( 
\frac{\beta \rho \omega ^{-j}}{\left( ab\right) ^{\frac{1}{2}}}\right) ^{k}%
\text{ \ \ \ \ \ \ \ \ \ \ \ \ \ \ \ \ \ \ \ \ \ \ \ \ \ \ \ \ \ \ \ \ \ \ \
\ \ \ \ \ \ \ \ \ \ \ \ \ \ \ }
\end{equation*}%
\ \ \ \ \ \ \ \ \ \ \ \ \ \ \ \ \ \ \ \ \ \ \ \ \ \ \ \ \ \ \ 
\begin{equation*}
=\frac{a}{\left( ab\right) ^{\frac{n-1}{2}}}\left( \mathbf{A}\frac{\alpha
^{n}r-\left( ab\right) ^{\frac{n}{2}}}{\alpha \rho \omega ^{-j}-\left(
ab\right) ^{\frac{1}{2}}}-\mathbf{B}\frac{\beta ^{n}r-\left( ab\right) ^{%
\frac{n}{2}}}{\beta \rho \omega ^{-j}-\left( ab\right) ^{\frac{1}{2}}}%
\right) \text{ \ \ \ \ \ \ \ \ \ \ \ \ \ \ \ \ \ \ \ \ \ \ \ \ \ \ \ \ \ \ \
\ \ \ \ \ }
\end{equation*}%
\begin{equation*}
=\frac{a}{\left( ab\right) ^{\frac{n-1}{2}}\left( \alpha \rho \omega
^{-j}-\left( ab\right) ^{\frac{1}{2}}\right) \left( \beta \rho \omega
^{-j}-\left( ab\right) ^{\frac{1}{2}}\right) }\left( r\rho \omega
^{-j}\left( \alpha \beta \right) \left( \mathbf{A}\alpha ^{n-1}-\mathbf{B}%
\beta ^{n-1}\right) \right.
\end{equation*}%
\begin{equation*}
\left. -r\left( ab\right) ^{\frac{1}{2}}\left( \mathbf{A}\alpha ^{n}-\mathbf{%
B}\beta ^{n}\right) -\rho \omega ^{-j}\left( ab\right) ^{\frac{n}{2}}\left( 
\mathbf{A}\beta -\mathbf{B}\alpha \right) +\left( ab\right) ^{\frac{n+1}{2}%
}\left( \mathbf{A}-\mathbf{B}\right) \right) .\text{ \ \ \ \ }
\end{equation*}%
If $k$ is even, we obtain%
\begin{equation*}
\lambda _{j}\left( W_{r}\right) =\frac{rw_{n}-w_{0}+\left( \frac{a}{b}%
\right) ^{^{\frac{1}{2}}}\rho \omega ^{-j}\left( rw_{n-1}+\left(
bw_{0}-w_{1}\right) \right) }{\rho ^{2}\omega ^{-2j}+\left( ab\right) ^{%
\frac{1}{2}}\rho \omega ^{-j}-1}.\text{ \ \ \ \ \ \ \ \ \ \ \ \ }
\end{equation*}%
If $k$ is odd, we get%
\begin{equation*}
\lambda _{j}\left( W_{r}\right) =\frac{\left( \frac{a}{b}\right) ^{^{\frac{1%
}{2}}}rw_{n}-w_{0}+\rho \omega ^{-j}\left( rw_{n-1}+\left( \frac{a}{b}%
\right) ^{^{\frac{1}{2}}}\left( bw_{0}-w_{1}\right) \right) }{\rho
^{2}\omega ^{-2j}+\left( ab\right) ^{\frac{1}{2}}\rho \omega ^{-j}-1}.\text{
\ }
\end{equation*}
\end{proof}

\begin{theorem}
The determinant of $W_{r}$ is 
\begin{equation*}
\det \left( W_{r}\right) =\frac{\left( w_{0}-\left( \frac{a}{b}\right) ^{%
\frac{\zeta \left( n\right) }{2}}rw_{n}\right) ^{n}-r\left( \left( \frac{a}{b%
}\right) ^{^{\frac{\zeta \left( n+1\right) }{2}}}rw_{n-1}+\left( \frac{a}{b}%
\right) ^{^{\frac{1}{2}}}\left( bw_{0}-w_{1}\right) \right) ^{n}}{1-\left( 
\frac{a}{b}\right) ^{\frac{\zeta \left( n\right) }{2}}p_{n}r+\left(
-1\right) ^{n}r^{2}}.
\end{equation*}
\end{theorem}

\begin{proof}
Since $\det \left( W_{r}\right) =\dprod\limits_{j=0}^{n-1}\lambda _{j}\left(
W_{r}\right) ,$ we get the desired result.
\end{proof}

\section{Acknowledgement}

This research is supported by Amasya University Scientific Research Project
Coordinatorship (BAP). Project No: FMB-BAP 20-0474.

\end{document}